\documentclass[a4paper,11pt]{amsart}
\usepackage{amssymb}
\usepackage{amsmath}
\usepackage{amsthm}
\usepackage[abbrev]{amsrefs}

\title[On semi-continuity problems for minimal log discrepancies]
{On semi-continuity problems for minimal log discrepancies}

\author{Yusuke Nakamura}
\address{Graduate School of Mathematical Sciences, 
the University of Tokyo, 3-8-1 Komaba, Meguro-ku, Tokyo 153-8914, Japan.}
\email{nakamura@ms.u-tokyo.ac.jp}

\theoremstyle{plain}
\newtheorem{thm}{Theorem}[section]
\newtheorem{prop}[thm]{Proposition}
\newtheorem{lem}[thm]{Lemma}
\newtheorem{cor}[thm]{Corollary}
\newtheorem{conj}[thm]{Conjecture}

\theoremstyle{definition}

\theoremstyle{remark}
\newtheorem{rmk}[thm]{Remark}

\newcommand{\mbA}{\mathbb{A}}
\newcommand{\mbN}{\mathbb{N}}
\newcommand{\mbQ}{\mathbb{Q}}
\newcommand{\mbR}{\mathbb{R}}
\newcommand{\mbC}{\mathbb{C}}
\newcommand{\mbL}{\mathbb{L}}
\newcommand{\mbZ}{\mathbb{Z}}

\newcommand{\mcD}{\mathcal{D}}
\newcommand{\mcJ}{\mathcal{J}}
\newcommand{\mcI}{\mathcal{I}}
\newcommand{\mcO}{\mathcal{O}}
\newcommand{\mcV}{\mathcal{V}}
\newcommand{\mcX}{\mathcal{X}}
\newcommand{\mcY}{\mathcal{Y}}
\newcommand{\mcZ}{\mathcal{Z}}
\newcommand{\mfa}{\mathfrak{a}}
\newcommand{\mfb}{\mathfrak{b}}
\newcommand{\mfc}{\mathfrak{c}}
\newcommand{\mfs}{\mathfrak{s}}
\newcommand{\mfA}{\mathfrak{A}}

\newcommand{\mfR}{\mathfrak{R}}

\DeclareMathOperator{\mld}{mld}
\DeclareMathOperator{\sht}{sht}

\DeclareMathOperator{\ord}{ord}
\DeclareMathOperator{\Jac}{Jac}

\DeclareMathOperator{\Spec}{Spec}
\DeclareMathOperator{\codim}{codim}

\DeclareMathOperator{\Hom}{Hom}
\DeclareMathOperator{\ideal}{ideal}

\DeclareMathOperator{\cent}{c}
\DeclareMathOperator{\Conj}{Conj}
\DeclareMathOperator{\Aut}{Aut}

\begin{document}
\begin{abstract}
We show the semi-continuity property of minimal log discrepancies for 
varieties which have a crepant resolution in the category of Deligne-Mumford stacks. 
Using this property, we also prove the ideal-adic semi-continuity problem for 
toric pairs. 
\end{abstract}


\maketitle

\section{Introduction}
The minimal log discrepancy (mld for short) was introduced by Shokurov, in order to reduce the conjecture of 
terminations of flips to a local problem about singularities. 
Recently, this has been a fundamental invariant in the minimal model program. 
There are two related conjectures about mld's, 
the ACC (ascending chain condition) conjecture and the LSC (lower semi-continuity) conjecture. 
Shokurov showed that these two conjectures imply the conjecture of terminations of flips \cite{Sho:letter}. 

In the first half of this paper, we consider the LSC conjecture, proposed by Ambro \cite{Amb2}.

\begin{conj}[LSC conjecture]\label{conj:LSC}
Let $(X, \Delta)$ be a log pair. Then, the function
\[
|X| \to \mbR _{\ge 0} \cup \{ - \infty \}; \qquad
x \mapsto  \mld _x (X, \Delta)
\]
is lower semi-continuous, where we denote by $|X|$ the set of all closed points of $X$. 
\end{conj}
In this conjecture,  $\mld _x (X, \Delta)$ denotes the minimal log discrepancy, 
and will be defined in section \ref{subsec:mld}. 
The value $\mld _x (X, \Delta)$ reflects the singularity of $X$ and $\Delta$ around $x$.

The LSC conjecture is known under some conditions. 
If $X$ is toric or $\dim X \le 3$, this conjecture is known by work of Ambro \cite{Amb}. 
If $X$ is smooth, this conjecture was proved by Ein, Musta\c{t}\u{a}, and Yasuda  \cite{EMY}. 
They used the description of mld's by the language of jet schemes. 
Ein and Musta\c{t}\u{a} extended this result to locally complete intersection varieties \cite{EM}. 

The first purpose of this paper is to prove the LSC conjecture on varieties with quotient singularities. 

Let $X$ be a $\mbQ$-Gorenstein normal variety, and assume that 
$X$ has a crepant resolution in the category of Deligne-Mumford stacks. 
Here, a proper birational morphism $f: \mcX \to X$ from a smooth Deligne-Mumford stack $\mcX$ is \textit{crepant} 
if $K_{\mcX} = f^* K_{X}$ holds. 

\begin{thm}\label{thm:main1}
Let $X$ be a $\mbQ$-Gorenstein normal variety. 
Assume $X$ has a crepant resolution in the category of Deligne-Mumford stacks. 
Then, the followings hold. 

\begin{itemize}
\item[(i)] 
Let $\Delta$ be an effective $\mbR$-Cartier divisor on $X$, and 
$\mfa$ an $\mbR$-ideal sheaf on $X$, that is, a formal product 
$\prod \mfa _i ^{r_i}$ of ideals $\mfa _i$ with $r_i$ positive real numbers. 
Then the function
\[
|X| \to \mbR _{\ge 0} \cup \{ - \infty \}; \qquad
x \mapsto  \mld _x (X, \Delta, \mfa)
\]
is lower semi-continuous. 

\item[(ii)] 
Let $T$ be a variety, $\Delta$ an effective $\mbR$-Cartier divisor on $X$, and $x$ a closed point of $X$. 
For an $\mbR$-ideal sheaf $\mfA$ on $X \times T$, 
the function
\[
|T| \to \mbR _{\ge 0} \cup \{ - \infty \}; \qquad
p \mapsto  \mld _x (X, \Delta, \mfA _p)
\]
is lower semi-continuous, where $\mfA _p$ is the restriction of $\mfA$ to $X \times \{ p \}$. 
\end{itemize}

Especially, Conjecture \ref{conj:LSC} holds for the variety $X$. 
\end{thm}

Assume that a finite group $G$ acts on a smooth variety $M$, and that this action is free in codimension $1$. 
Then the quotient variety $M/G$ and the quotient stack $[M/G]$ are isomorphic in codimension $1$. 
Hence, $M/G$ has a crepant resolution $[M/G] \to M/G$. 
Therefore we get the following corollary. 

\begin{cor}\label{cor:quote}
Conjecture \ref{conj:LSC} holds if $X$ has only quotient singularities. 
\end{cor}

To prove Theorem \ref{thm:main1}, 
we employ Yasuda's theory of twisted jet stacks \cite{Yas:1}, \cite{Yas:2}. 
We take a crepant resolution $f: \mcX \to X$, and 
describe $\mld _x$ by the dimensions of twisted jet stacks on $\mcX$. 

As a corollary, we have an application to terminations. 
We can show that there is no infinite sequence of flips 
if we start from a symplectic variety with only quotient singularities. 
This is an extension of Matsushita's result \cite[Proposition 2.1, Lemma 2.1]{Mat}. 
In the statement below, we consider log minimal model programs. 
For a log pair $(X, D)$, a \textit{$(K_X + D)$-MMP} is a log minimal model program which contracts a $(K_X + D)$-negative extremal ray repeatedly.
The \textit{$(K_{X_i} + D_i)$-flip} is the flip corresponding to a flipping contraction of a $(K_{X_i} + D_i)$-negative extremal ray
(See \cite{KM} for more details). 
Note that $K_X$ is trivial if $X$ is a symplectic variety.

\begin{cor}\label{cor:symplectic}
Let $X$ be a symplectic variety with only terminal and quotient singularities, 
and $D_0$ an effective $\mbR$-Cartier divisor on $X$. 
Then, there is no infinite sequence of flips
\[
X=X_0 \overset{f_0}{\dashrightarrow} X_1 \overset{f_1}{\dashrightarrow} X_2 \overset{f_2}{\dashrightarrow} \cdots
\]
such that $f_i$ is a $(K_{X_i} + D_i)$-flip, where $D_i$ is the $\mbR$-Cartier divisor defined as $D_{i+1} := f_i(D_i)$ inductively. 

Especially, if the linear system $|D_0|$ has no fixed divisor, any $(K_X + D_0)$-MMP terminate. 
\end{cor}

In the latter half of this paper, we consider the following related conjecture proposed by Musta\c{t}\u{a}.

\begin{conj}[Musta\c{t}\u{a}]\label{conj:mustata}
Let $(X, \Delta)$ be a log pair, $Z$  a closed subset of $X$, and $I_Z$  its ideal sheaf. 
Let $\mfa = \prod _{i=1} ^s \mfa _i ^{r_i}$ be an $\mbR$-ideal sheaf with 
ideal sheaves $\mfa _1, \ldots , \mfa _s$ on $X$ and positive real numbers $r_1 , \ldots , r_s$. 

Then there exists a positive integer $l$ such that the following holds: 
if ideal sheaves $\mfb _1, \ldots , \mfb _s$ satisfy $\mfa _i + I_Z ^l = \mfb _i + I_Z ^l$, then 
\[
\mld _Z (X, \Delta, \mfa) = \mld _Z (X, \Delta, \mfb)
\]
holds, where we put $\mfb := \prod _{i=1} ^s \mfb _i ^{r_i}$. 
\end{conj}

This conjecture is related to Shokurov's ACC conjecture on mld's by generic limits of ideals \cite[Remark 2.5.1]{Kaw:klt}. 
Generic limits were introduced by Koll\'ar \cite{Kol:which}, and using this method, 
de Fernex, Ein, and Musta\c{t}\u{a} proved 
Shokurov's ACC conjecture for log canonical thresholds \cite{dFEM:acc_lct}. 
The generic limit is one of constructions of a limit of a sequence of ideals. 
This idea of considering the limit of ideals originates in de Fernex and Musta\c{t}\u{a} \cite{dFM:limit_lct}, 
and they constructed a limit using non-standard analysis.

Musta\c{t}\u{a}'s conjecture is known under some conditions. 
If 
$
\mld _Z (X, \Delta, \mfa) = 0 
$ 
holds,
the conjecture is known  
by work of de Fernex, Ein, and Musta\c{t}\u{a} \cite{dFEM:acc_lct}. 
If the triple $(X, \Delta, \mfa)$ is Kawamata log terminal around $Z$, 
then the conjecture is known by work of Kawakita \cite{Kaw:klt}. 
The remaining case is 
when $(X, \Delta, \mfa)$ is log canonical around $Z$ and satisfies 
$
\mld _Z (X, \Delta, \mfa) > 0 
$. 
The conjecture in dimension $2$ was proved by Kawakita \cite{Kaw:surface}.

The second purpose of this paper is to prove Musta\c{t}\u{a}'s conjecture on varieties with a $\mbC ^*$-action. 
Let $X = \Spec A$ be an affine variety with a $\mbC ^*$-action, and $A = \bigoplus _{m \in \mbZ} A^{(m)}$ 
the induced graded ring structure. 
Then, the action of $\mbC ^*$ on $X= \Spec A$ is said to be \textit{of ray type} if 
$A^{(m)} = 0$ holds for all $m <0$ or $A^{(m)} = 0$ holds for all $m >0$. 
In the following proposition, we assume that the action is of ray type. 

\begin{prop}\label{prop:main2}
Let $X$ be a $\mbQ$-Gorenstein normal affine variety, 
$\Delta$ an effective $\mbR$-Cartier divisor, 
$\mfa = \prod _{i=1} ^s \mfa _i ^{r_i}$ be an $\mbR$-ideal sheaf on $X$, and 
$Z$ a closed subset of $X$. 
Suppose that $\mbC ^*$ acts on $X$ and assume the following conditions:

\begin{itemize}
\item The variety $X$ has a $\mbC ^*$-equivariant crepant resolution in the category of Deligne-Mumford stacks. 
\item The $\mbC ^*$-action on $X$ is of ray type.
\item All the components of $\Delta$ and $\mfa _i$ are $\mbC ^*$-invariant. 
\item $Z$ is the set of all $\mbC ^*$-fixed points in $X$. 
\end{itemize}

Then, 
Conjecture \ref{conj:mustata} holds for $(X, \Delta, \mfa)$ and $Z$. 
\end{prop}

\begin{rmk}\label{rmk:ideal_b}
In the above proposition, we assume that the ideal $\mfa _i$ is $\mbC ^*$-invariant, but 
the ideal $\mfb _i$ is not necessarily $\mbC ^*$-invariant. 
\end{rmk}

As an application, we can prove Musta\c{t}\u{a}'s conjecture for toric varieties. 

\begin{thm}\label{thm:toric}
Let $X$ be a normal toric variety, $(X, \Delta)$ a log pair, and  
$Z$ a closed subset of $X$. 
Assume that $\Delta$ and $Z$ are torus invariant. 
Then, for an $\mbR$-ideal sheaf $\mfa = \prod _{i=1} ^s \mfa _i ^{r_i}$ with torus invariant ideal sheaves $\mfa _i$, 
Conjecture \ref{conj:mustata} holds 
for $(X,\Delta, \mfa)$ and $Z$.
\end{thm}

\begin{rmk}\label{rmk:ideal_b2}
As Remark \ref{rmk:ideal_b}, in the above theorem, we assume that the ideal $\mfa _i$ is torus invariant, but 
the ideal $\mfb _i$ in Conjecture \ref{conj:mustata} is not necessarily torus invariant. 
Therefore the theorem cannot follow directly from a combinatorial description 
of mld's for toric triples. 
\end{rmk}

In the proof of Proposition \ref{prop:main2}, showing the inequality 
\[
\mld _Z (X, \Delta, \mfa) 
\le \mld _Z (X, \Delta, \mfb)
\]
is essential (see Remark \ref{rmk:one_direction}). 
To prove this inequality, we consider a degeneration of the ideal $\mfb$. 
We explain the idea of the proof below. 
For the sake of simplicity, we assume $s = 1$ 
and denote $\mfa ' := \mfa _1$, $\mfb ' := \mfb _1$, and $r:=r_1$. 
When $l$ is sufficiently large, we can degenerate $\mfb '$ to some ideal $\mfb ' _0$ which contains $\mfa '$. 
Hence we get $\mld _Z (X, \Delta, (\mfa ') ^r) \le \mld _Z (X, \Delta, (\mfb ' _0) ^r)$, and 
the proposition reduces to the inequality 
$\mld _Z (X, \Delta, (\mfb ' _0) ^r) \le \mld _Z (X, \Delta, (\mfb ') ^r)$. 
This inequality follows from the semi-continuity property like Theorem \ref{thm:main1} (ii). 
However, the above inequality does not directly follow from Theorem \ref{thm:main1} (ii). 
It is because 
$Z$ has possibly positive dimension. Hence, we need to look the construction of the above degeneration.

The paper is organized as follows. 
In section 2, we review the definition of mld's and the theory of Yasuda's twisted jet stacks. 
Some lemmas necessary for the proof of Theorem \ref{thm:main1} are also proved in section 2. 
In section 3, we prove Theorem \ref{thm:main1}. 
In section 4, we prove Proposition \ref{prop:main2} and Theorem \ref{thm:toric}. 

\subsection*{Notation and convention}
Throughout this paper, we work over the field of complex numbers $\mbC$. 

\begin{itemize}
\item We denote by $\mbN$ the set of all non-negative integers. 
\item Every Deligne-Mumford stack in this paper is separated. 
\item For a Deligne-Mumford stack $\mcX$, we denote by $|\mcX|$ the set of all $\mbC$-valued points. 
\item For a morphism $f : \mcX \to T$ from a Deligne-Mumford stack to a variety $T$, 
and a closed point $p \in |T|$, 
we denote by $\mcX _p$ the fiber of $f$ over $p$. 
For an ideal sheaf $\mfa$ on $\mcX$, we denote by $\mfa _p$ the restriction to the fiber $\mcX _p$. 
In addition, for a morphism $g: \mcX \to \mcY$ between Deligne-Mumford stacks over $T$, 
we denote by $g_p: \mcX _p \to \mcY _p$ the induced morphism between the fibers. 
\end{itemize}

\section{Preliminaries}
\subsection{Minimal log discrepancies}\label{subsec:mld}

We recall the notations in the theory of singularities in the minimal model program. 

A \textit{log pair} $(X, \Delta)$ is a normal variety $X$ and an effective $\mbR$-divisor $\Delta$ such that $K_X + \Delta$ is 
$\mbR$-Cartier. 
An \textit{$\mbR$-ideal sheaf} of $X$ is a formal product $\mfa _1 ^{r_1} \cdots \mfa _s ^{r_s}$, 
where $\mfa _1 , \ldots , \mfa _s $ are ideal sheaves on $X$ and $r_1 , \ldots , r_s$ are positive real numbers. 
For a log pair $(X, \Delta)$ and an $\mbR$-ideal sheaf $\mfa$, we call $(X, \Delta, \mfa)$ a \textit{log triple}. 
When $\Delta = 0$, we sometimes drop $\Delta$ and write $(X, \mfa)$. 
If $Y_i$ is the closed subscheme of $X$ corresponding to $\mfa _i$, 
we sometimes identify the triple 
$(X, \Delta, \sum _{i=1} ^s r_i Y_i)$ with the triple $(X, \Delta, \prod _{i=1} ^s \mfa _i ^{r_i})$. 

For a proper birational morphism $f: X' \to X$ from a normal variety $X'$ 
and a prime divisor $E$ on $X'$, the \textit{log discrepancy} of $(X, \Delta, \mfa)$ 
at $E$ is defined as
\[
a_E (X, \Delta, \mfa) := 1 + \ord _E (K_{X'} - f^* (K_X + \Delta)) -\ord _E \mfa, 
\]
where $\ord _E \mfa := \sum _{i=1} ^s r_i \ord _E \mfa _i$. 
The image $f(E)$ is called the \textit{center of $E$ on $X$}, and we denote it by $\cent _X (E)$. 
For a closed subset $Z$ of $X$, the \textit{minimal log discrepancy} (mld for short) over $Z$ is 
defined as 
\[
\mld _Z (X, \Delta, \mfa) := \inf _{\cent _X(E) \subset Z} a_E (X, \Delta, \mfa). 
\]
In the above definition, 
the infimum is taken over all prime divisors $E$ on $X'$ with the center $\cent _X (E) \subset Z$, 
where $X'$ is a higher birational model of $X$, that is, 
$X'$ is the source of some proper birational morphism $X' \to X$. 

\begin{rmk}\label{rmk:attain}
It is known that $\mld _Z (X, \Delta, \mfa)$ is in $\mbR _{\ge 0} \cup \{ - \infty \}$ and that
if $\mld _Z (X, \Delta, \mfa) \ge 0$, then the infimum on right hand side in the definition is actually 
the minimum. 
\end{rmk}

\begin{rmk}\label{rmk:property}
Let $(X, \Delta, \mfa)$ and $Z$ be as above. Mld's have the following properties. 
\begin{enumerate}
\item[(i)] If $Z_1, \ldots , Z_c$ are the irreducible components of $Z$, then
\[
\mld _Z (X, \Delta, \mfa) = \min _{1 \le i \le c} \mld _{Z_i} (X, \Delta, \mfa).
\]

\item[(ii)] If $U_1, \ldots , U_c$ are open subsets of $X$ such that 
$Z \subset \bigcup _{j=1} ^c U_j$, then
\[
\mld _Z (X, \Delta, \mfa) = \min _{1 \le j \le c} \mld _{Z \cap U_j} (U_j, \Delta |_{U_j}, \mfa |_{U_j}),
\]
where $\mfa |_{U_j} := \prod _{i=1} ^s (\mfa _i \mcO _{U_j}) ^{r_i}$. 
\end{enumerate}
These properties easily follow from the definition. 
\end{rmk}

\subsection{Notation and remarks on Deligne-Mumford stacks}

In this section, we review some properties of Deligne-Mumford stacks (DM stacks for short). 
In this paper, we are mainly interested in separated DM stacks of finite type over complex number $\mbC$. 

Let $\mcX$ be a DM stack of finite type over $\mbC$. 
We can consider the set of $\mbC$-points $|\mcX|$ of $\mcX$, and 
it admits a Zariski topology \cite{LM}. 
Keel and Mori \cite{KM:quot} proved that the coarse moduli space $X$ exists for $\mcX$. 
Then, the induced map $|\mcX| \to |X|$ is a homeomorphism. 

It is known that a DM stack is \'etale locally isomorphic to a quotient stack (see for instance \cite{KM:quot}). 
Let $X$ be the coarse moduli space of $\mcX$. 
Then, there exists an \'etale covering $\{ X_i \to X \}_i$ such that for every $i$, 
the \'etale pullback $\mcX \times _{X} X_i$ is isomorphic to a quotient stack $[M_i / G_i]$, 
where $M_i$ is a variety over $\mbC$ and $G_i$ is a finite group. 

We prove two lemmas about DM stacks. 
The first one is about families of DM stacks. 

\begin{lem}\label{lem:fiber}
Let $M$ be a variety and $G$ a finite group. Suppose $G$ acts on $M$. 
Let $T$ be a variety and $f: [M/G] \to T$ be a morphism of stacks. 
Then, for a closed point $p \in |T|$, the fiber $[M/G]_p$ is isomorphic to the quotient stack $[M_p /G]$. 
\end{lem}

\begin{proof}
First, we remark that $M_p$ is $G$-invariant because the morphism $M \to T$ factors through the quotient stack $[M/G]$. 

Let $n$ be the order of the group $G$. 
It follows that both of the morphisms $M_p \to [M/G]_p$ and $M_p \to [M_p / G]$ are 
surjective finite \'etale morphisms of degree $n$. 
Since $M_p \to [M/G]_p$ factors through $M_p \to [M_p / G]$, 
we can conclude that $[M/G]_p$ is isomorphic to $[M_p / G]$. 
\end{proof}

The second lemma is about semi-continuity of the dimensions of fibers on a family of DM stacks. 

\begin{lem}\label{lem:semiconti}
Let $\mcX$ be a DM stack of finite type over $\mbC$, and
$f: \mcX \to T$ a morphism from $\mcX$  to a variety $T$. 
Then, for every $n \in \mbN$, the set 
\[
|\mcX|_{\ge n} := \{ x \in |\mcX|  \, | \, \dim _x f^{-1} (f(x)) \ge n \}
\]
is closed in $|\mcX|$. In the above equation, we denote by $\dim _x$ 
the dimension around $x$. 
\end{lem}

\begin{proof}
If $\mcX$ is a variety, then this lemma is well-known (see for instance \cite{Gortz}). 

Let $X$ be the coarse moduli space of $\mcX$. 
By the universality of the coarse moduli space, 
$f: \mcX \to T$ factors through $\mcX \to X$. 
Since the induced map $|\mcX| \to |X|$ is a homeomorphism, 
the assertion follows from the case of varieties. 
\end{proof}

\subsection{Yasuda's twisted jet stacks}\label{sec:jet_stack}
Twisted jet stacks were introduced by Yasuda \cite{Yas:1}, \cite{Yas:2}. 
First, we recall the definition. 

Let $\mcX$ be a DM stack over $\mbC$. 
For an affine scheme $S = \Spec R$ over $\mbC$ and a non-negative integer $n$, 
we denote 
\[
D_{n, S} := \Spec R[t]/(t^{n+1}).
\]
For a positive integer $l$, we denote by $\mu _l$ the cyclic group of $l$-th roots of the unity. 
We consider the natural group action $\mu _l$ on $D_{nl, S}$ induced by the following ring homomorphism:
\[
R[t]/(t^{nl+1}) \to R[t]/(t^{nl+1}) \otimes \mbC[x]/(x^l -1); \qquad t \mapsto t \otimes x. 
\]
Then, we denote
\[
\mcD _{n, S} ^l := [D _{nl, S}/ \mu _l]. 
\]

A \textit{twisted $n$-jet of order $l$} of $\mcX$ over $S$ is a representable morphism 
$\mcD _{n, S} ^l \to \mcX$. 
Yasuda defined the \textit{stack of twisted $n$-jets of order $l$}, and we denote it by $\mcJ _n ^l \mcX$. 
An object of $\mcJ _n ^l \mcX$ over $S$ is a twisted $n$-jet of order $l$. 
For a morphism $f: S \to T$ over $\mbC$, we have an induced morphism $f': \mcD _{n, S} ^l \to \mcD _{n, T} ^l$. 
A morphism in $\mcJ _n ^l \mcX$ from $\gamma : \mcD _{n, S} ^l \to \mcX$ to $\gamma ' : \mcD _{n, T} ^l \to \mcX$
is a $2$-morphism from $\gamma$ to $\gamma ' \circ f'$. 
The stack $\mcJ _n \mcX$ of \textit{twisted $n$-jets} is 
the disjoint union of the stacks $\bigsqcup _{l \ge 1} \mcJ _n ^l \mcX$.
Both $\mcJ _n ^l \mcX$ and $\mcJ _n \mcX$ are actually DM stacks \cite[Theorem 18]{Yas:2}. 
When $X$ is a scheme, $\mcJ _n ^l X = \emptyset$ holds for $l \ge 2$, and 
$\mcJ _n X$ can be identified with the usual $n$-th jet scheme $J_n X$.

For $0 \le n_1 \le n_2$, we have the \textit{truncation morphism}
$\varphi ^{\mcX} _{n_2, n_1}: \mcJ _{n_2} \mcX \to \mcJ _{n_1} \mcX$. 
This corresponds to the surjective ring homomorphism $R[t]/(t^{n_2 l +1}) \to R[t]/(t^{n_1 l +1})$. 
Since $\varphi ^{\mcX} _{n_2, n_1}$ is an affine morphism, we have a projective limit and projections
\[
\mcJ _{\infty} \mcX := \lim _{\leftarrow} \mcJ _{n} \mcX, \qquad 
\varphi ^{\mcX} _{\infty, n}: \mcJ _{\infty} \mcX \to \mcJ _{n} \mcX. 
\]

The stack $\mcJ _{0} \mcX$ can be identified with the inertia stack and the projection
\[
\varphi ^{\mcX}_{0, \text{b}}: \mcJ _{0} \mcX \to \mcX
\]
is a finite morphism. 
Here, we use ``$\text{b}$'' in the index as an abbreviation of the word ``base''. 
We denote by $\varphi ^{\mcX} _{n, \text{b}}$ the composite map 
\[
\mcJ _{n} \mcX \xrightarrow{\varphi ^{\mcX} _{n, 0}} 
\mcJ _{0} \mcX 
\xrightarrow{\varphi ^{\mcX} _{0, \text{b}}} \mcX
\]
for $n \in \mbN \cup \{ \infty \}$.

In the theory of jet schemes, the $\mbC ^*$-action and the relative jet schemes
can be defined \cite[Section 2]{Mus:jet}. 
In the following part of this section, we generalize these concepts to twisted jet stacks. 

We already defined the truncation morphism 
$\varphi ^{\mcX} _{n, 0}: \mcJ _{n} \mcX \to \mcJ _{0} \mcX$. 
On the other hand, we also have the \textit{zero section} 
$\sigma_{n} ^{\mcX} : \mcJ _{0} \mcX \to \mcJ _{n} \mcX$. 
This is defined by the composition with the stack morphism $\mcD _{n, S} ^l \to \mcD _{0, S} ^l$
induced by the ring inclusion $R \hookrightarrow R[t]/(t^{nl+1})$. 
We will write simply $\varphi _{m,n}$ (resp. $\sigma _{n}$) 
instead of $\varphi ^{\mcX} _{m,n}$ (resp. $\sigma ^{\mcX} _{n}$) 
when no confusion can arise. 

The twisted jet stack $\mcJ _{n} \mcX$ admits the $\mbC ^*$-action which is induced by the following $\mbC ^*$-action on $R[t]/(t^{nl+1})$: 
\[
R[t]/(t^{nl+1}) \to R[t]/(t^{nl+1}) \otimes \mbC [s, s^{-1}]; \qquad 
t \to t \otimes s. 
\]
The following lemma is used in the proof of Theorem \ref{thm:main1}. 

\begin{lem}\label{lem:closed_abs}
Let $W$ be a $\mbC ^*$-invariant closed subset of $|\mcJ _{n} \mcX|$. 
Then $\varphi _{n,0} (W)$ is a closed subset of $|\mcJ _{0} \mcX|$. 
\end{lem}

\begin{proof}
The $\mbC ^*$-action on $\mcJ _{n} \mcX$ is induced by the action $\mbC ^* \times \mcD _{n, S} ^l \to \mcD _{n, S} ^l$. 
By the definition above, the morphism $\mbC ^* \times \mcD _{n, S} ^l \to \mcD _{n, S} ^l$ 
is uniquely extended to the morphism
$
\mbC \times \mcD _{n, S} ^l \to \mcD _{n, S} ^l
$. 
Therefore, the $\mbC ^*$-action $\mbC ^* \times \mcJ _{n} \mcX \to \mcJ _{n} \mcX$ is extended to the morphism
\[
\psi:\mbC \times \mcJ _{n} \mcX \to \mcJ _{n} \mcX. 
\]
By definition, 
for any $\alpha \in |\mcJ _{n} \mcX|$, 
we have $\psi (0, \alpha) = \sigma _{n}(\varphi _{n,0} (\alpha))$. 

It is sufficient to show the equality $\varphi _{n,0} (W) = \sigma _{n} ^{-1} (W)$ 
because $\sigma _{n} ^{-1} (W)$ is closed.
Fix an element $\alpha \in W$. 
Since $W$ is $\mbC ^*$-invariant, $\psi (\gamma, \alpha) \in W$ for any $\gamma \in \mbC ^*$. 
Because $W$ is closed, $\psi (0, \alpha) \in W$ holds. 
Since $\psi (0, \alpha) = \sigma _{n} (\varphi _{n,0} (\alpha))$, 
we have the equality $\varphi _{n,0} (W) = \sigma _{n} ^{-1} (W)$. 
\end{proof}

Next, we construct the \textit{relative twisted jet stacks}.
\begin{lem}\label{lem:relative}
Let $f: \mcY \to T$ be a morphism from a DM stack $\mcY$ to a variety $T$. 
Then there exist a DM stack $\mcJ _{n} (\mcY/T)$ and a morphism $g: \mcJ _{n} (\mcY/T) \to T$ such that 
for any point $p \in |T|$, 
\[
\mcJ _{n} (\mcY/T)_p \cong \mcJ _{n} (\mcY _p)
\]
holds, where $\mcJ _{n} (\mcY/T)_p$ is the fiber of $g$ over $p$, and $\mcY _p$ is the fiber of $f$ over $p$. 
\end{lem}

For the proof of Lemma \ref{lem:relative}, we use an \'etale local description of twisted jet stacks \cite[Proposition 16]{Yas:2}. 

Let $[M/G]$ be a quotient stack with scheme $M$ and finite group $G$. 
Fix a positive integer $l$ and an embedding $a: \mu _l \hookrightarrow G$. 
We denote by $J_n M$ the $n$-th jet scheme of $M$, then $\mu _l$ acts on $J_n M$ in two ways. 
First, the action $\mu _l \curvearrowright D_{n,S}$ induces an action $\mu _l \curvearrowright J_n M$. 
On the other hand, we have the action $G \curvearrowright J_n M$, and this action induces an action
$\mu _l \curvearrowright J_n M$ by the embedding $a: \mu _l \hookrightarrow G$. 
We define $J_n ^{(a)} M$ to be the closed subscheme of $J_n M$ where the above two actions are identical. 
Then we have a concrete description of twisted jet stacks:
\[
\mcJ _{n} ^l \mcX \cong \bigsqcup _{a \in \Conj (\mu _l , G)} [J_{nl} ^{(a)} M / C_a],  
\]
where $\Conj (\mu _l , G)$ is the conjugacy classes of embeddings $\mu _l \hookrightarrow G$, 
and $C_a$ is the centralizer of $a$. 
\begin{proof}[Proof of Lemma \ref{lem:relative}]
When $\mcY = Y$ is a scheme, the \textit{relative jet scheme} $J_n (Y/T)$ exists and 
\[
J_n (Y/T)_p \cong J_n (Y_p)
\]
holds \cite[Section 2]{Mus:jet}. 
Besides, $J_n (Y/T)$ can be characterized by the following representability: 
\[
\Hom _{\text{Sch}/ T} (Z, J_n (Y/T)) \cong 
\Hom _{\text{Sch}/ T} \big( Z \times _{\mbC} \Spec ( \mbC [t]/(t^{n+1}) ) , Y \big)
\]
for every scheme $Z$ over $T$. 

First, we consider the case when $\mcY$ is a quotient stack $[M/G]$. 
Fix a positive integer $l$ and an embedding $a: \mu _l \hookrightarrow G$. 
Then, two actions $\mu _l \curvearrowright J_n (M/T)$ are 
induced by the actions $\mu _l \curvearrowright \Spec ( \mbC [t]/(t^{m+1}))$ and $\mu _l \curvearrowright M$. 
These actions are compatible with above two $\mu _l$-actions on $J_n (M_p)$ when we restrict them to the fibers $J_n (M/T)_p$. 
Set $J_n ^{(a)} (M/T)$ to be the closed subscheme of $J_n (M/T)$ where the above two actions are identical. 

We put 
\[
\mcZ ^l := \bigsqcup _{a \in \Conj (\mu _l , G)} [J_{nl} ^{(a)} (M/T) / C_a]. 
\]
Restricting this to the fiber over $p \in |T|$, we have 
\[
\mcZ ^l _p \cong \bigsqcup _{a \in \Conj (\mu _l , G)} [J_{nl} ^{(a)} (M_p) / C_a] 
\cong \mcJ _{n} ^l ([M/G]_p)
\]
by Lemma \ref{lem:fiber}. 
Hence, 
$\bigsqcup _{l \ge 1} \mcZ ^l$ satisfies the property in the statement. 

For the general case, 
we introduce the relative twisted jet stack as a category. 
We define $\mcJ ^l _{n} (\mcY/T)$ as follows. 
For a scheme $S$ over $T$, an object of $\mcJ ^l _{n} (\mcY/T)$ over $S$ is a 
representable morphism $\mcD _{n, S} ^l \to \mcY$ over $T$. 
For a morphism $f: S_1 \to S_2$ over $T$, we have an induced morphism $f': \mcD _{n, S_1} ^l \to \mcD _{n, S_2} ^l$. 
A morphism in $\mcJ ^l _{n} (\mcY/T)$ from $\gamma : \mcD _{n, S_1} ^l \to \mcY$ to $\gamma ' : \mcD _{n, S_2} ^l \to \mcY$
is a $2$-morphism from $\gamma$ to $\gamma ' \circ f'$. 
Then, the category $\mcJ ^l _{n} (\mcY/T)$ is actually a stack by the same reason as in the absolute case 
\cite[Lemma 12]{Yas:2}. 
We define $\mcJ _{n} (\mcY/T) = \bigsqcup _{l \ge 1} \mcJ ^l _{n} (\mcY/T)$. 

To complete the proof, it is sufficient to show that for a quotient stack $\mcY = [M/G]$, 
\[
\mcJ _{n} ^l ([M/G]/T) = \bigsqcup _{a \in \Conj (\mu _l , G)} [J_{nl} ^{(a)} (M/T) / C_a]
\]
holds. 
This equation can be confirmed by following the proof for the absolute case 
\cite[Proposition 16]{Yas:2}. 
Here, we use the property of the relative jet schemes about the representability as remarked above. 
\end{proof}

For the relative twisted jet stacks, we can also define the truncation morphism 
$\varphi ^{\mcY /T} _{n_2,n_1} : \mcJ _{n_2}(\mcY /T) \to \mcJ _{n_1}(\mcY /T)$ , 
the zero section $\sigma ^{\mcY /T} _n : \mcJ _{0}(\mcY /T) \to \mcJ _{n}(\mcY /T)$, and 
$\mbC ^*$-action on $\mcJ _{n}(\mcY /T)$. 
If we restrict them to the fiber over a closed point $p \in T$, 
these definitions are compatible with the definitions in the absolute case. 
We also have the following lemma. 

\begin{lem}\label{lem:closed}
Let $W$ be a $\mbC ^*$-invariant closed subset of $|\mcJ _{n} (\mcY/T)|$. 
Then $\varphi ^{\mcY/T} _{n,0} (W)$ is a closed subset of $|\mcJ _{0} (\mcY/T)|$. 
\end{lem}

\begin{proof}
We have the equality
$\varphi ^{\mcY /T} _{n,0} (W) = (\sigma ^{\mcY /T} _{n}) ^{-1} (W)$
by the proof of Lemma \ref{lem:closed_abs}. 
The right hand side is clearly closed in $|\mcJ _{0} (\mcY/T)|$. 
\end{proof}

\subsection{Motivic integration}
First, we define the space in which motivic integration takes value. 
We introduce the notion of the \textit{Grothendieck semiring}, following Yasuda \cite[Section 3]{Yas:2}.

In this section, we fix a positive integer $r$. 
A \textit{convergent stack} is the pair $(\mcX, \alpha)$ of a DM stack $\mcX$ and a function
\[
\alpha: \{\text{connected component of $\mcX$} \} \to  \frac{1}{r} \mbZ
\]
satisfying the following two conditions:

\begin{itemize}
\item $\mcX$ has at most countably many connected components and the all connected components are of finite type, and 
\item for every integer $n$, there are at most finitely many connected components $\mcV$ such that 
$\dim \mcV + \alpha (\mcV) > n
$.
\end{itemize}

A DM stack $\mcX$ of finite type is identified with the convergent stack $(\mcX, 0)$. 
For two convergent stacks $(\mcX , \alpha)$ and $(\mcY, \beta)$, 
a \textit{morphism} $f: (\mcX , \alpha) \to (\mcY, \beta)$ of convergent stacks is 
a morphism $g: \mcX \to \mcY$ of stacks satisfying $\beta = \alpha \circ f$. 
We denote by $(\mfR ^{1/r}) '$ the set of the isomorphism classes of convergent stacks. 
$(\mfR ^{1/r}) '$ admits a semiring structure by the disjoint union $\sqcup$ and the product $\times$. 
For $(\mcX , \alpha) \in (\mfR ^{1/r}) '$, we can define the \textit{dimension} $\dim (\mcX , \alpha)$ by 
\[
\max \{ \dim \mcV + \alpha (\mcV) \, | \, \text{$\mcV$ is a connected component of $\mcX$} \}. 
\]
For empty set $\emptyset$, we define $\dim \emptyset := - \infty$. 

For each $n \in \mbZ$, we can define a relation $\sim _n$ on $(\mfR ^{1/r}) '$ to be the strongest equivalence relation
satisfying the following three relations:

\begin{itemize}
\item If $(\mcX, \alpha)$ and $(\mcY, \beta)$ are convergent stacks with $\dim (\mcY, \beta) < n$, 
then $(\mcX, \alpha) \sim _n (\mcX, \alpha) \sqcup (\mcY, \beta)$.
\item If $(\mcX, \alpha)$ is a convergent stack and $\mcY$ is a closed substack of $\mcX$, 
then $(\mcX, \alpha) \sim _n (\mcY, \alpha |_{\mcY}) \sqcup (\mcX \setminus \mcY , \alpha)$. 
\item Let $(\mcX, \alpha)$ and $(\mcY, \beta)$ be convergent stacks. Assume there exists a representable morphism of stacks
$f: \mcX \to \mcY$ such that 
$f^{-1}(x) \cong \mbA ^{\beta (x) - \alpha (f^{-1}(x))}$ holds for any point $x \in |\mcY|$. 
Then $(\mcX, \alpha) \sim _n (\mcY, \beta)$. 
\end{itemize}

Then we can define new relation $\sim$ on $(\mfR ^{1/r}) '$ as follows: 
For any $a, b \in (\mfR ^{1/r}) '$, $a \sim b$ if and only if $a \sim _n b$ for any integer $n$. 
For a convergent stack $(\mcX, \alpha) \in (\mfR ^{1/r}) '$, we denote by $\{ (\mcX, \alpha) \}$ 
the equivalence classes of $\mcX$. Besides, we denote by $\mfR ^{1/r}$ the set of all the equivalence classes. 
By definition, the map
\[
\dim: \mfR ^{1/r} \to \frac{1}{r}\mbZ \cup \{ - \infty \}, \qquad \{ (\mcX, \alpha) \} \mapsto \dim (\mcX, \alpha)
\]
is well defined, and a semiring structure on $\mfR ^{1/r}$ is induced by the semiring structure on $(\mfR ^{1/r}) '$. 

Next, we introduce the motivic measure on a smooth DM stack \cite[Section 2]{Yas:2}. 

Let $\mcX$ be a smooth DM stack of finite type and pure dimension $d$. 
For a non-negative integer $n$, a subset $A \subset |\mcJ _{\infty} \mcX|$ is called an \textit{$n$-cylinder} 
if $A = \varphi _{\infty, n} ^{-1} \varphi _{\infty, n} (A)$ and 
$\varphi _{\infty, n} (A)$ is a constructible subset of $|\mcJ _{n} \mcX|$. 
For an $n$-cylinder $A \subset |\mcJ _{\infty} \mcX|$ we define 
\[
\mu _{\mcX} (A):= \{ \varphi _{\infty, n} (A) \} \mbL ^{-nd} \in \mfR ^{1/r}, 
\]
where we denote $\{ \mbA ^1 \}$ by $\mbL$.
This definition is independent of $n$ \cite[Lemma 43]{Yas:2}. 
A subset $A \subset |\mcJ _{\infty} \mcX|$ is called \textit{negligible} 
if there exist cylinders $C_i$ for $i \ge 1$ such that $A = \bigcap _{i \ge 1} C_i$ and 
$\lim _{i \to \infty} \codim C_i = \infty$ hold, 
where we denote $\codim C := \codim _{|\mcJ _n \mcX|} \varphi _{\infty, n} (C)$ for an $n$-cylinder $C$. 

Let $A \subset |\mcJ _{\infty} \mcX|$ be a subset. 
A function $F : A \to \mfR ^{1/r}$ is called \textit{measurable} if
there exist a negligible subset $B$ and countably many cylinders $A_i$ such that
$A = B \sqcup \bigsqcup _{i \ge 1} A_i$ and $F$ is constant on $A_i$. 
For such $F$, we define the \textit{motivic integration} as follows:
\[
\int _{A} F d \mu _{\mcX} := \sum_{i \ge 1} F(A_i) \cdot \mu _{\mcX} (A_i) \in \mfR ^{1/r} \cup \{ \infty \}, 
\]
where it takes value $\infty$ if 
there exist infinitely many $i$ with $\dim F(A_i) \cdot \mu _{\mcX} (A_i) >m$ 
for some integer $m$.

Next, we define the motivic measure on singular varieties \cite{DL}, \cite{Yas:2}. 

Let $X$ be a variety of dimension $d$. 
For a non-negative integer $n$, a subset $A \subset |J _{\infty} X|$ is called \textit{stable} at level $n$
if 
\begin{itemize}
\item $A = \varphi _{\infty, n} ^{-1} \varphi _{\infty, n} (A)$, 
\item $\varphi_{\infty, n} (A)$ is a constructible subset of $|J _{n} X|$, and
\item for any $m \ge n$ the truncation morphism 
$\varphi _{m,n} : \varphi _{\infty, m}(A) \to \varphi _{\infty, n}(A)$
is a piecewise trivial fibration with fibers $\mbA ^{(m-n)d}$. 
\end{itemize}
For such a subset $A \subset |J _{\infty} X|$, we define 
\[
\mu _{X} (A):= \{ \varphi _{\infty, n} (A) \} \mbL ^{-nd} \in \mfR ^{\frac{1}{r}}. 
\]
A subset $A \subset |J _{\infty} X|$ is called \textit{negligible} 
if there exist constructible subsets 
$C_i \subset \varphi  _{\infty, i} (J _{\infty} X)$ for $i \ge 0$ such that $A = \bigcap _{i \ge 0} C_i$ and 
$\lim _{i \to \infty} \dim C_i -di = - \infty$ hold. 
As in the case of smooth DM stacks, 
we can define the motivic integration on singular varieties, replacing $n$-cylinders by stable subsets at level $n$ 
in the above definition.

\subsection{The transformation rule}\label{subsec:trans}
Yasuda proved the transformation rule for a proper birational morphism from a smooth DM stack to a variety. 

First, we define the \textit{shift function} \cite{Yas:2}. 

Let $\mcX$ be a smooth DM stack of dimension $d$, $x \in |\mcX|$ a $\mbC$-valued point, and 
$a: \mu _l \hookrightarrow \Aut (x)$ an embedding. 
Then the cyclic group $\mu _l$ acts on the tangent space $T_x \mcX$ by the embedding $a$, and 
induces a decomposition $T_x \mcX = \bigoplus _{i = 1} ^l T_{x,i}$, where 
$T_{x,i}$ is the eigenspace on which 
$\xi _l := e^{2\pi \sqrt{-1} / l} \in \mu _l$ acts by the multiplication of $\xi _l ^i$. 
Then we define 
\[
\sht(a) := d - \frac{1}{l} \sum _{i=1} ^l i \cdot \dim T_{x,i} \in \mbQ. 
\]
Since the space $|\mcJ _{0} \mcX|$ is set theoretically equal to 
$\bigsqcup _{l \ge 1} \bigsqcup _{x \in |\mcX|} \Conj (\mu _l, \Aut (x))$, 
we can write a point of $|\mcJ _{0} \mcX|$ by a pair $(x, a)$ with $x \in |\mcX|$ and $a \in \Conj (\mu _l, \Aut (x))$. 
Then we can define the shift function on $|\mcJ _{0} \mcX|$ by followings
\[
\sht: |\mcJ _{0} \mcX| \to \mbQ; \qquad (x, a) \mapsto \sht (a). 
\]
This function is constant on each connected component $\mcV$ of $|\mcJ _{0} \mcX|$. 
We define the shift function $\mfs _{\mcX}$ on $|\mcJ _{\infty} \mcX|$ by the composite map
\[
|\mcJ _{\infty} \mcX| \xrightarrow{\varphi _{\infty, 0}} |\mcJ _{0} \mcX| 
\xrightarrow{\sht} \mbQ. 
\]
So we have a measurable function
\[
\mbL ^{\mfs _{\mcX}} : |\mcJ _{\infty} \mcX| \to \mfR ^{1/r}
\]
for some sufficiently divisible positive integer $r$. 

Next, we define the \textit{order function} along an ideal. 
Let $\mcX$ be a DM stack, $\mcY$ a closed substack of $\mcX$, and $\mcI$ its ideal sheaf on $\mcX$. 
For $\alpha \in |\mcJ _{\infty} \mcX| \setminus |\mcJ _{\infty} \mcY|$, we define a rational number $\ord _{\mcI} \alpha$ as follows. 
The twisted jet $\alpha$ can be written by a morphism $\alpha : [ \Spec \mbC [[t]] / \mu _l ] \to \mcX$ for some $l \ge 1$. 
Hence we have the lift $\overline{ \alpha} : \Spec \mbC [[t]] \to \mcX$ and have 
$\overline{ \alpha} ^{-1} : \mcO _{\mcX} \to \mbC [[t]]$. Let $m$ be the integer satisfying 
$\overline{ \alpha} ^{-1} \mcI = (t^m)$. Then we define $\ord _{\mcI} \alpha := \frac{m}{l}$. 

Let $f: \mcX \to X$ be a morphism from a DM stack to a variety. 
Then $f$ induces a map $f_{\infty} : |\mcJ _{\infty} \mcX| \to |J_{\infty} X|$ as follows: 
For a twisted jet $\mcD _{n, S} ^l \to \mcX$, we have the composition $\mcD _{n, S} ^l \to \mcX \to X$. 
Since $D_{n,S}$ is the coarse moduli space of $\mcD _{n, S} ^l$, 
the above map uniquely factors through $\mcD _{n, S} ^l \to D_{n,S}$, 
and we have an $n$-th jet $D_{n,S} \to X$. 
Since the map $\mcD _{n, S} ^l \to D_{n,S}$ is defined by $t \mapsto t ^l$, 
the order functions on $\mcX$ and $X$ are compatible with $f_{\infty}$. 
That is, 
for an ideal sheaf $I$ on $X$, $\ord _{I \mcO _{\mcX}} = \ord _I \circ f_{\infty}$ holds. 

Here, we can state Yasuda's transformation rule for a proper birational morphism $f: \mcX \to X$ 
from a smooth DM stack to a variety. 
The \textit{Jacobian ideal sheaf} of $f$ is defined to be the 
$0$-th Fitting ideal sheaf of $\Omega _{\mcX / X}$, and we denote by $\Jac _f$. 

\begin{thm}[\cite{Yas:2}]\label{thm:trans}
Let $X$ be a variety, $\mcX$ a smooth DM stack, and $f: \mcX \to X$ a proper birational morphism. 
Let $A$ be a subset of $|J _{\infty} X |$ and $F: A \to \mfR ^{1/r}$ a measurable function. 
Then $F \circ f_{\infty}: f_{\infty} ^{-1} (A) \to \mfR ^{1/r}$ is measurable, and 
\[
\int _A F d\mu _{X} = 
\int _{f_{\infty} ^{-1} (A)} (F \circ f_{\infty}) \mbL ^{- \ord _{\Jac _f} + \mfs _{\mcX}} 
d \mu _{\mcX} \in \mfR ^{1/r} \cup \{ \infty \}
\]
holds. 
\end{thm}

\subsection{Minimal log discrepancies and jet schemes}
In \cite{EMY}, Ein, Musta\c{t}\u{a}, and Yasuda showed 
that mld's can be described by the language of jet schemes. 
Suppose $X$ is a $\mbQ$-Gorenstein variety with Gorenstein index $r$ and 
set $n:= \dim X$. 
Then, we have a natural map
\[
(\Omega _X ^n )^{\otimes r} \longrightarrow \mcO _X (r K_X).
\]
Since $\mcO _X (r K_X)$ is an invertible sheaf, 
we have an ideal sheaf $J_{r ,X}$ such that the image of this map is $J_{r ,X} \mcO _X (r K_X)$. 

\begin{thm}[\cite{EMY}]\label{thm:EMY}
Let $X$ be a normal, $d$-dimensional $\mbQ$-Gorenstein variety with Gorenstein index $r$, 
$Y_1, \ldots , Y_s$ proper closed subschemes of $X$, and $W$ a proper closed subset of $X$. 
If $q_1, \ldots , q_s$ are positive real numbers, then 
\[
\mld _W (X, \sum _{i=1} ^s q_i Y_i) 
= \inf _{m \in \mbN ^s}  \Big \{ 
d - \sum _{i=1} ^s q_i m_i - 
\dim \int _{A_m} \mbL ^{\frac{1}{r} \ord_{J_{r,X}}} d \mu _X
\Big \}, 
\]
where $A_m := \varphi _{\infty , 0} ^{-1} (W) \cap \bigcap _{1 \le i \le s} \ord ^{-1} _{I_{Y_i}} ( \ge m_i) 
\subset J_{\infty} X $. 

Moreover, if the mld is finite, then the infimum on the right-hand side is actually a minimum, and 
the minimum is attained by some $m \in S$, where $S$ is a finite subset of $\mbN ^s$ and 
only depends on the numerical data of a log resolution of $(X, \sum _{i=1} ^s q_i Y_i)$ and $W$. 
If the mld is infinite, then a negative value is attained by some $m \in S$. 
\end{thm}

\begin{rmk}\label{rmk:s}
The finite set $S$ can be constructed as follows. 
Take a log resolution $\pi : X' \to X$ of the log pair $(X, \sum _{i=1} ^s q_i Y_i)$ and $W$. 
Let $D_1, \ldots , D_c$ be the divisors on $X'$ satisfying $\pi (D_j) \subset W$. 
Then we can take $S$ as
\[
S = \{ m=(\ord _{D_j} Y_1, \ldots , \ord _{D_j} Y_s) \in \mbN ^s \mid 1 \le j \le c \}. 
\]
\end{rmk}

\section{Lower semi-continuity problems on varieties with quotient singularities}
\subsection{Minimal log discrepancies and twisted jet stacks}
Let $X$ be a $d$-dimensional $\mbQ$-Gorenstein variety with Gorenstein index $r$, 
$\mfa$ an $\mbR$-ideal sheaf, and $W$ a closed subset of $X$. 
Assume that there exists a crepant resolution $f : \mcX \to X$ in the category of DM stacks. 
In this setting, we can describe $\mld _W (X, \mfa)$ by a motivic integration on $\mcX$. 

The $\mbR$-ideal sheaf $\mfa$ can be written by $\mfa = \prod _{i=1} ^s \mfa _i ^{q_i}$ 
with ideal sheaves $\mfa _i$ and positive real numbers $q_i$. 
By Theorem \ref{thm:EMY}, 
\[
\mld _W (X, \mfa) = 
\inf _{m \in \mbN ^s}  \Big \{ 
d - \sum _{i=1} ^s q_i m_i - 
\dim \int _{A_{W,m}} \mbL ^{\frac{1}{r} \ord_{J_{r,X}}} d \mu _X
\Big \}, 
\]
where $A_{W,m} := \varphi _{\infty , 0} ^{-1} (W) \cap \bigcap _{1 \le i \le s} \ord ^{-1} _{\mfa _i} ( \ge m_i) 
\subset |J_{\infty} X| $. 

We apply Theorem \ref{thm:trans} to the resolution $f$. 
We write $\mfa ' _i := \mfa _i \mcO _{\mcX}$. 

\begin{lem}\label{lem:mld_stack}
Let $X, \mcX, \mfa, \mfa'$ and $W$ be as above. Then we have 
\[
\mld _W (X, \mfa) = 
\inf _{m \in \mbN ^s}  \Big \{ 
d - \sum _{i=1} ^s q_i m_i - 
\dim \int _{A ' _{W,m}} \mbL ^{\mfs _{\mcX}} d \mu _{\mcX}
\Big \}, 
\]
where we denote
\[
A ' _{W,m} := (f \circ \varphi _{\infty , \text{b}}) ^{-1} (W) \cap 
\bigcap _{1 \le i \le s} \ord ^{-1} _{\mfa ' _i} ( \ge m_i ) 
\subset |\mcJ _{\infty} \mcX|.
\]
\end{lem}

\begin{proof}
First, we easily have
$
f_{\infty} ^{-1} (A_{W,m}) = A ' _{W,m}
$.

By the definition of the Jacobian ideal $\Jac _f$, 
the image of the canonical morphism $f ^* \Omega ^d _{X} \to \Omega ^d _{\mcX}$ is 
equal to $\Jac _f \otimes \Omega ^d _{\mcX}$. 
Let $r$ be the Gorenstein index of $X$. By the definition of $J_{r, X}$, 
the image of the canonical morphism
$
(\Omega _X ^d) ^{\otimes r} \to \mcO _X (r K_X)
$
is equal to $J_{r, X} \otimes \mcO _X (r K_X)$. 
Therefore we have an equation $\Jac _f ^r = J_{r, X} \mcO _{\mcX} (- r K_{\mcX / X}) $. 
Since $f$ is crepant in this case, we have
\[
\frac{1}{r} \ord _{J_{r, X}} \circ f_{\infty} = \ord _{\Jac _f}. 
\]
By Theorem \ref{thm:trans}, we can conclude 
\[
\int _{A_{W,m}} \mbL ^{\frac{1}{r} \ord_{J_{r,X}}} d \mu _X
= 
\int _{A' _ {W,m}} \mbL ^{\mfs _{\mcX}} d \mu _{\mcX}, 
\]
which completes the proof. 
\end{proof}

\subsection{Proof of Theorem \ref{thm:main1}}
We prove a more general statement. 

\begin{thm}\label{thm:LSC_general}
Let $f : X \to T$ be a flat surjective morphism between varieties, $\mfA$ an $\mbR$-ideal sheaf on $X$, and
$W$ a closed subset of $X$. 
Assume the following three conditions: 
\begin{itemize}
\item $W$ is proper over $T$ and each $W_p$ is a proper subset of $X_p$, 
\item each $X_p$ is irreducible, reduced, normal, and $\mbQ$-Gorenstein, and
\item there exist a DM stack $\mcX$ and a morphism $g: \mcX \to X$ such that for each closed point $p \in T$, 
the induced morphism $g_p : \mcX _p \to X _p$ is a crepant resolution. 
\end{itemize}
Then, the function
\[
|T| \to \mbR \cup \{ \pm \infty \}; \qquad 
p \mapsto \mld _{W_p} (X_p, \mfA _p)
\]
is lower semi-continuous. 
\end{thm}

To prove this theorem, we borrow an idea from the argument in \cite[Proposition 2.3]{Mus:jet}. 
We need two lemmas. 
First one is about the set $S$ in Remark \ref{rmk:s}. 

\begin{lem}\label{lem:s}
Let $f : X \to T$, $\mfA$, and $W$ be as in Theorem \ref{thm:LSC_general}. 
In addition, suppose that $\mfA _p \not= 0$ for any $p \in |T|$. 
Then, for each $p \in |T|$, 
we can take a finite set $S_p \subset \mbN ^s$ in Theorem \ref{thm:EMY} 
for $(X_p, \mfA _p)$ and $W_p$ such that $\bigcup _{p \in T} S_p$ is also a finite set. 
\end{lem}

\begin{proof}
Take a log resolution $g: X' \to X$ of $(X, \mfA)$ and $W$. 
Let $D = \bigcup_{i \in I} D_i$ be the union of the exceptional divisors and 
the supports of $\mfA \mcO _{X'}$ and $g^{-1}(W)$. 
By generic smoothness, 
we can take a non-empty open set $U \subset T$ such that $X'$ is smooth over $U$ 
and $\bigcap _{i \in I'} D_i$ is also smooth over $U$ for any subset $I' \subset I$. 
Take $V \subset X$ be a non-empty open set over which $g$ is isomorphic. 
Since $f$ is flat and surjective, $f(V)$ is also open. 
Then, for any $p \in U \cap f(V)$, the restriction to the fiber $g_p : X' _p \to X_p$ 
is a log resolution of $(X _p, \mfA _p)$ and $W_p$. 
Hence, we can take $S_p$ uniformly for $p \in U \cap f(V)$. 
Since $U \cap f(V)$ is a non-empty open subset of $T$, the assertion follows from the induction on $\dim T$. 
\end{proof}

Second lemma is about the shift function on a family. 
For a smooth morphism 
$f:\mcX \to T$ from a DM stack to a variety, 
the relative $0$-th twisted jet stack $\mcJ _0 (\mcX / T)$ can be defined. 
An element of $|\mcJ _0 (\mcX / T)|$ can be written by 
$(p, \alpha)$ with $p \in |T|$ and $\alpha \in \mcJ _0 (\mcX _p)$. 
On $|\mcJ _0 (\mcX / T)|$, the shift function $\sht$ can be defined by 
$\sht (p, \alpha ) = \sht (\alpha)$. 

\begin{lem}\label{lem:shift}
Let $f: \mcX \to T$ be as above. 
For a connected component $\mcV$ of $|\mcJ _0 (\mcX / T)|$, 
$\sht$ is constant on $\mcV$. 
\end{lem}

\begin{proof}
We may assume that $\mcX$ is a quotient stack $[M/G]$. 
Then the relative twisted jet stack can be written by
\[
\mcJ _{n} ([M/G]/T) = \bigsqcup _{l \ge 1} \bigsqcup _{a \in \Conj (\mu _l , G)} [J_{nl} ^{(a)} (M/T) / C_a]. 
\]
Take $l \ge 1$ and $a : \mu _l \hookrightarrow G$ such that $\mcV \subset [J_{nl} ^{(a)} (M/T) / C_a]$. 
Then, for every $x \in |\mcV|$, $a : \mu _l \hookrightarrow G$ factors through $\Aut x \subset G$. 

Let $T_{\mcX / T}$ be the relative tangent bundle on $\mcX$ over $T$, and 
$i : \mcV \to \mcX$ the projection. 
Then $\mu _l$ acts on $i ^* T_{\mcX / T}$ over $T$, 
and this action is fiberwise compatible with the $\mu _l$-action on $T_{\mcX _x}$ in Section \ref{subsec:trans}. 
This action induces a decomposition $i ^* T_{\mcX / T} = \bigoplus _{i=1} ^l T_i$, where 
$T_{i}$ is the vector bundle on which 
$\xi _l \in \mu _l$ acts by the multiplication of $\xi _l ^i$. 
It follows that for each $i$, $\dim T_{x,i}$ is constant on $|\mcV|$, which completes the proof. 
\end{proof}

\begin{proof}[Proof of Theorem \ref{thm:LSC_general}]
If $W_p = \emptyset$, then $\mld _{W_p} (X_p, \mfA _p) = \infty$. 
If $W_p \not = \emptyset$ and $\mfA _p = 0$, then $\mld _{W_p} (X_p, \mfA _p) = - \infty$. 
Since $W$ is proper over $T$, 
\[
T' := f(W) \cap \{ p \in |T| \mid \mfA _p = 0 \}
\] is closed in $|T|$. 
Replacing $T$ by $T \setminus T'$, we may assume that $\mfA _p \not= 0$ for any $p \in |T|$. 

Since $\mfA$ is an $\mbR$-ideal sheaf, it can be written by 
$\mfA = \prod _{i=1} ^s \mfA _i ^{q_i}$ with ideal sheaves $\mfA _i$ and 
positive real numbers $q_i$. 
We denote $\mfA ' _i := \mfA _i \mcO _{\mcX}$ and 
denote by $(\mfA ' _i)_p$ the restriction of ideal $\mfA ' _i$ to the fiber $\mcX _p$. 

For $p \in |T|$ and $m \in \mbN ^s$, we put 
\[
B_{p, m}:= 
(g_p \circ \varphi ^{\mcX _p} _{\infty , \text{b}}) ^{-1} (W_p) \cap 
\bigcap _{1 \le i \le s} \ord ^{-1} _{(\mfA ' _i)_p} ( \ge m_i) . 
\]
Take a finite set $S_p \subset \mbN ^s$ for each closed point $p \in T$ as in Lemma \ref{lem:s}. 
Fix a multi-index $m \in \bigcup _{p \in |T|} S_p$. 
By Lemma \ref{lem:mld_stack}, it is sufficient to prove that
the function 
\[
|T| \to \mbQ; \qquad p \mapsto \dim \int _{B _{p, m}} \mbL ^{\mfs _{\mcX _p}} d \mu _{\mcX _p}
\]
is upper semi-continuous. 

Let $\mcY _i$ be the closed substack of $\mcX$ corresponding to the ideal sheaf $\mfA ' _i$. 
We identify the relative twisted jet stack $\mcJ _m (\mcY _i / T)$ with a closed substack 
of $\mcJ _m (\mcX / T)$, and 
we also identify $B _{p,m}$ with a closed subset of $\mcJ _{\infty} (\mcX / T)_p$. 

For a connected component $\mcV$ of $|\mcJ _{0} ( \mcX /T )|$, 
the shift function is constant on $\mcV$ by Lemma \ref{lem:shift}. 
Since $|\mcJ _{0} ( \mcX /T )|$ has only finitely many connected components, 
it is sufficient to prove that the function
\[
|T| \to \mbZ; \qquad p \mapsto \dim \mu _{\mcX} \big( (\varphi ^{\mcX / T}  _{\infty, 0}) ^{-1} (\mcV) \cap B _{p, m} \big)
\]
is upper semi-continuous for each $\mcV$. 
Take a positive integer $m'$ such that $m' \ge m_i$ for any $i$. 
Since $B _{p, m}$ is an $m'$-cylinder, we have
\begin{align*}
\mu _{\mcX} \big( (\varphi _{\infty, 0} ^{\mcX / T}) ^{-1} (\mcV) \cap B _{p, m} \big) = 
\{
S \cap (f \circ g \circ \varphi _{m' , \text{b}} ^{\mcX / T}) ^{-1} (p)
\}\mbL ^{-m'd}, 
\end{align*}
where we put 
\begin{align*}
S := (\varphi_{m',0} ^{\mcX / T}) ^{-1} & (\mcV) \cap 
( \varphi_{m' ,\text{b}} ^{\mcX / T}) ^{-1} (g^{-1}(W)) \\
& \cap 
\bigcap _{1 \le i \le s} (\varphi _{m' ,m_i } ^{\mcX / T})^{-1} (|\mcJ _{m_i} (\mcY _i /T)|). 
\end{align*}
Let $F$ be the composite map 
\[
S \hookrightarrow |\mcJ _{m'} (\mcX /T) | \xrightarrow{\varphi_{m', 0} ^{\mcX / T}} 
|\mcJ _0 (\mcX /T)| \xrightarrow{\varphi_{0,\text{b}} ^{\mcX / T}} |\mcX| \xrightarrow{g} |X| \xrightarrow{f} |T|.
\]
Then this theorem reduces to the upper semi-continuity property of the function
\[
|T| \to \mbZ; \qquad p \mapsto \dim F^{-1} (p). 
\]
For each integer $n$, we set

\begin{align*}
S_{\ge n} &:= \{ s \in S \, | \, \dim _s F^{-1} ( F(s)) \ge n  \}, \\
|T|_{\ge n} &:= \{ p \in |T| \, | \, \dim F^{-1} (p) \ge n \}. 
\end{align*}
Then we have $F (S_{\ge n}) = |T|_{\ge n}$. 
We need to show that $|T|_{\ge n}$ is closed in $|T|$, but
we know only that $S_{\ge n}$ is a closed subset of $S$ by Lemma \ref{lem:semiconti}. 
The space $|\mcJ _{m'} (\mcX /T)|$ admits a $\mbC ^*$-action (see Section \ref{sec:jet_stack}). 
Each $|\mcJ _{m_i} (\mcY _i /T)|$ is a $\mbC ^*$-invariant closed subset of $|\mcJ _{m'} (\mcX /T)|$, and 
each fibre of the morphism $\varphi _{m' , 0} ^{\mcX / T}$ is invariant on this action. 
Hence $S_{\ge n}$ is a $\mbC ^*$-invariant closed subset of $|\mcJ _{m'} (\mcX /T)|$. 
By Lemma \ref{lem:closed}, $\varphi _{m', 0} ^{\mcX / T} (S_{\ge n})$ is closed in $|\mcJ _{0} (\mcX /T)|$. 
Since both of the maps $\varphi_{0,\text{b}} ^{\mcX / T}$ and $g$ are closed, 
the set 
\[
|X|_{\ge n} := (g \circ \varphi_{0,\text{b}} ^{\mcX / T}) (\varphi _{m', 0} ^{\mcX / T} (S_{\ge n}))
\]
is also closed in $|X|$. 
Because $|X|_{\ge n}$ is a subset of $|W|$ and $W$ is proper over $T$, we can conclude the set 
\[
|T|_{\ge n} = f (|X|_{\ge n})
\]
is also closed in $|T|$, which completes the proof. 
\end{proof}

\begin{proof}[Proof of Theorem \ref{thm:main1}]
We prove (i) first. Since $X$ is $\mbQ$-Gorenstein, 
we may assume $\Delta = 0$ by forcing $\Delta$ to $\mfa$. 
For $i = 1, 2$, we denote by $p_i : X \times X \to X$ the $i$-th projection. 
Since $\mfa$ is an $\mbR$-ideal sheaf, it can be written by $\mfa = \prod _{i=1} ^s \mfa _i ^{q_i}$ 
with $\mbR$-ideal sheaves $\mfa _i $ and positive real numbers $q_i$. 
Set $\mfA = \prod _{i=1} ^s (p_1 ^* \mfa _i)^{q_i}$, and $W$ be the diagonal set of $X \times X$. 
Applying Theorem \ref{thm:LSC_general} to 
the morphism $p_2 : X \times X \to X$, the ideal $\mfA$, and the closed set $W$, 
we have the assertion in (i). 

Next we prove (ii). 
Since $X$ is $\mbQ$-Gorenstein, 
we may assume $\Delta = 0$ by forcing $\Delta \times T$ to $\mfA$. 
Applying Theorem \ref{thm:LSC_general} to 
the projection $X \times T \to T$, the ideal $\mfA$, and the closed set $\{ x \} \times T$, 
we have the assertion in (ii). 
\end{proof}

Now, we can prove Corollary \ref{cor:symplectic}. 
Matsushita proved the same statement when $X$ is a smooth symplectic variety and $D_0$ 
is a $\mbQ$-divisor \cite[Proposition 2.1, Lemma 2.1]{Mat}. 
First, by Theorem \ref{thm:main1} (1), 
this can be extended to the case when $X$ is a symplectic variety with only quotient singularities. 
In addition, Kawakita's result about the ACC conjecture \cite{Kaw:discreteness} can 
extend it to the case when $D_0$ is an $\mbR$-divisor. 
\begin{proof}[Proof of Corollary \ref{cor:symplectic}]
Let $Y$ be a normal variety and $\Gamma$ be a finite subset of $[0,1]$. 
Then, we define a set $A(X, \Gamma)$ as 
\[
A(X, \Gamma) := \{ \mld _x (X, \Delta) \mid \text{$(X,\Delta)$ is a log canonical pair, 
and $\Delta \in \Gamma$}\}, 
\]
here we write $\Delta \in \Gamma$ if all coefficients of $\Delta$ are in $\Gamma$. 

Let $\Gamma$ be the set of all coefficients of $D_0$. 
By the proof of \cite[Theorem]{Sho:letter}, 
it is sufficient to show that the set $\bigcup _{i \ge 0} A(X_i, \Gamma)$ satisfies the ascending chain condition 
and that the LSC conjecture holds for each $X_i$. 

By \cite[Main Theorem, Proposition 1]{Nam}, since $X_0$ has only terminal singularities, 
$X_0$ and $X_1$ can deform to a variety $X'$ by locally trivial deformations at the same time. 
So, inductively, 
we can say that $X_i$ has only quotient singularities and that $A(X_i, \Gamma) = A(X, \Gamma)$. 
By \cite[Theorem 1.2]{Kaw:discreteness}, we already know that $A(X, \Gamma)$ is finite. 
Hence $\bigcup _{i \ge 0} A(X_i, \Gamma)$ is also finite. 
On the other hand, since $X_i$ has only quotient singularities, 
the LSC conjecture holds for each $X_i$ by Theorem \ref{thm:main1}. 
Thus, we are done. 
\end{proof}

\section{The ideal-adic semi-continuity problem on varieties with a $\mbC ^{*}$-action}
\subsection{Proof of Proposition \ref{prop:main2}}
Before we start to prove Proposition \ref{prop:main2}, 
we provide some general remarks on Conjecture \ref{conj:mustata}. 

\begin{rmk}\label{rmk:modification}
Let $(X,\Delta, \mfa)$ and $Z$ be as in Conjecture \ref{conj:mustata}, and $f: X' \to X$ a proper birational morphism 
from a $\mbQ$-Gorenstein variety $X'$. 
We denote $\mfa \mcO _{X'} := \prod _{i=1} ^s (\mfa _i \mcO _{X'})^{r_i}$. 
Suppose $\Delta ' := f^* (K_X + \Delta ) - K_{X'}$ is effective. 
Then,  
Conjecture \ref{conj:mustata} holds for $(X,\Delta, \mfa)$ and $Z$
if the conjecture holds for $(X',\Delta ', \mfa \mcO _{X'})$ and $f^{-1}(Z)$. 

This follows from that 
$\mld_{Z} (X, \Delta , \mfc) = \mld_{f^{-1}(Z)} (X', \Delta ', \mfc \mcO _{X'})$ 
holds for any $\mbR$-ideal sheaf $\mfc$ and 
that $\mfa _i + I_Z ^l = \mfb _i + I_Z ^l$ implies 
$\mfa _i \mcO _{X'} + I_{f^{-1}(Z)} ^l = \mfb _i \mcO _{X'} + I_{f^{-1}(Z)} ^l$. 
\end{rmk}

\begin{rmk}\label{rmk:one_direction}
Let $\prod _{i=1} ^s \mfa _i ^{r_i}$ be an $\mbR$-ideal sheaf, $f:X' \to X$ a proper birational morphism, 
and $E$ a divisor on $X'$. 
If $\cent _X (E) \subset Z$, there exists a positive integer $l_E$ such that: 
if ideal sheaves $\mfb _1, \ldots , \mfb _s$ satisfy $\mfa _i + I_Z ^{l_E} = \mfb _i + I_Z ^{l_E}$, then 
\[
a _E (X, \Delta, \prod _{i=1} ^s \mfa _i ^{r_i}) =a_E (X, \Delta, \prod _{i=1} ^s \mfb _i ^{r_i})
\]
holds. 
In fact, if we take $l_E$ such that $\ord _E I_Z ^{l_E} > \ord _E \mfa _i$ for each $i$, 
then $\ord _E \mfa _i = \ord _E \mfb _i$ holds, and 
this implies the above equation. 

By this remark and Remark \ref{rmk:attain}, 
we can conclude
the inequality 
\[
\mld _Z (X, \Delta, \prod _{i=1} ^s \mfa _i ^{r_i}) \ge \mld _Z (X, \Delta, \prod _{i=1} ^s \mfb _i ^{r_i})
\]
in Conjecture \ref{conj:mustata}. 
In fact, by Remark \ref{rmk:attain}, 
\[
\mld _Z (X, \Delta, \prod _{i=1} ^s \mfa _i ^{r_i}) = a _E (X, \Delta, \prod _{i=1} ^s \mfa _i ^{r_i})
\]
holds for some $E$ 
when $\mld _Z (X, \Delta, \prod _{i=1} ^s \mfa _i ^{r_i})$ is non-negative. 
Then, by the above remark and the definition of mld's, 
\[
a _E (X, \Delta, \prod _{i=1} ^s \mfa _i ^{r_i}) =a_E (X, \Delta, \prod _{i=1} ^s \mfb _i ^{r_i}) 
\ge \mld _Z (X, \Delta, \prod _{i=1} ^s \mfb _i ^{r_i})
\]
holds for any $l > l_E$. 
When $\mld _Z (X, \Delta, \prod _{i=1} ^s \mfa _i ^{r_i})$ is negative, by Remark \ref{rmk:attain}, 
$a _E (X, \Delta, \prod _{i=1} ^s \mfa _i ^{r_i})$ is negative for some $E$, and we can continue the same argument. 
\end{rmk}

Let $X = \Spec A$ be an affine variety with a $\mbC ^*$-action, 
and $A = \bigoplus _{m \in \mbZ} A^{(m)}$ be the induced graded ring structure, where
\[
A^{(m)} := \{ f \in A \mid \gamma \cdot f = \gamma ^m f \ \text{for any $\gamma \in \mbC ^*$} \}.
\]
In what follows, we assume that the $\mbC ^*$-action is of ray type, and $A = \bigoplus _{m \ge 0} A^{(m)}$. 

For an element $f \in A$, 
we get the unique expression $f = \sum _{m \ge 0} f^ {(m)}$ with $f^ {(m)} \in A^{(m)}$.
Then we define $\deg f$ and $f \circ t$ as follows: 
\[
\deg f := \min \{ m \mid f^{(m)} \not = 0 \}, \quad
f \circ t := \sum _m t^m f^{(m)} \in A[t]. 
\]
For an ideal $I \subset A$, we define $\widetilde{I} \subset A[t]$ as
\[
\widetilde{I} := \ideal \left ( \frac{f \circ t}{t^{\deg f}} \, \Big| \, f \in A \setminus \{ 0 \} \right ). 
\]
For $\gamma \in \mbC$, the restriction $\widetilde{I} _{\gamma}$ is defined as the ideal
\[
\widetilde{I} _{\gamma} := \{ f({\gamma}) \, | \, f(t) \in \widetilde{I} \} \subset A. 
\]
It is clear that $\widetilde{I} _1 = I$. 
For ${\gamma} \in \mbC ^*$, we have the ring automorphism
\[
\phi _{\gamma}: A \longrightarrow A; \quad  
\sum _m f^{(m)} \mapsto \sum _m {\gamma}^{m} f^{(m)} , 
\]
and $\phi _{\gamma} (\widetilde{I} _1) = \widetilde{I} _{\gamma}$ holds. 

\begin{proof}[Proof of Proposition \ref{prop:main2}]
Since $X$ is $\mbQ$-Gorenstein, we may assume $\Delta = 0$ by forcing $\Delta$ to $\prod _{i=1} ^s \mfa _i ^{r_i}$. 
By exchanging $t$ with $t^{-1}$ in $\mbC ^* = \Spec \mbC [t, t^{-1}]$, 
we may assume $A = \bigoplus _{m \ge 0} A^{(m)}$. 
Since $Z$ is the set of all $\mbC ^*$-fixed points, it follows that $I_Z = \bigoplus _{m \ge 1} A^{(m)}$. 
Since ideals $\mfa _1, \ldots , \mfa _s$ are $\mbC ^*$-invariant, they are homogeneous. 
Let $f_{i 1} , \ldots , f_{i k_i}$ be homogeneous elements in $A$ such that
they generate $\mfa _i$. 
Set $l_i := 1+ \max_ {1 \le j \le k_i} \{ \deg f_{i j}  \}$. 

By Remark \ref{rmk:one_direction}, it is sufficient to show that
if ideal sheaves $\mfb _1, \ldots , \mfb _s$ satisfy $\mfa _i + I_Z ^{l_i} = \mfb _i + I_Z ^{l_i}$, then 
\[
\mld _Z (X, \prod _{i=1} ^s \mfa _i ^{r_i}) 
\le \mld _Z (X, \prod _{i=1} ^s \mfb _i ^{r_i})
\]
holds. 

We first prove that $\mfa _i \subset (\widetilde{\mfb _i})_0$. 
Since $\mfa _i = \ideal (f_{i 1} , \ldots , f_{i k_i}) \subset \mfb _i + I_Z ^{l_i}$, 
there exists $h_{i j}$ for each $1 \le j \le k_i$ such that
\[
f_{i j} + h_{i j} \in \mfb _i, \qquad h_{i j} \in I_Z ^{l_i}. 
\]
In respect to the degrees, we have $\deg f_{i j} \le l_i -1 < \deg h_{i j}$. 
Since $f_{i j}$ is homogeneous, we have $f_{i j} \in (\widetilde{\mfb _i})_0$. 
We thus get $\mfa _i \subset (\widetilde{\mfb _i})_0$. 

This inclusion implies the inequality
\[
\mld _Z (X, \prod _{i=1} ^s \mfa _i ^{r_i}) 
\le \mld _Z (X, \prod _{i=1} ^s (\widetilde{\mfb _i})_0 ^{r_i}).
\]
Because $(\widetilde{\mfb _i})_1 = \mfb _i$, it is sufficient to show that
\[
\mld _Z (X, \prod _{i=1} ^s (\widetilde{\mfb _i})_0 ^{r_i}) \le 
\mld _Z (X, \prod _{i=1} ^s (\widetilde{\mfb _i})_1 ^{r_i}). 
\]
On the other hand, we have the ring automorphism $\phi _{\gamma} : A \to A$ for $\gamma \in \mbC ^*$. 
By the definition of $\widetilde{\mfb _i}$, we have 
$\phi _{\gamma} ((\widetilde{\mfb _i})_1) = ((\widetilde{\mfb _i})_{\gamma})$. 
Hence we have
\[
\mld _Z (X, \prod _{i=1} ^s (\widetilde{\mfb _i})_1 ^{r_i}) 
=
\mld _Z (X, \prod _{i=1} ^s (\widetilde{\mfb _i})_{\gamma} ^{r_i})
\]
for every $\gamma \in \mbC ^*$. 
Therefore, the proof can be completed by showing that
the function 
\[
|\mbA ^1| \to \mbR \cup \{ - \infty \}; \qquad
p \mapsto \mld _Z (X, \prod _{i=1} ^s (\widetilde{\mfb _i})_p ^{r_i}) 
\]
is lower semi-continuous. 

In order to prove the semi-continuity, we consider the relative twisted jet stacks. 
Take a $\mbC ^*$-equivariant crepant resolution $\mcX \to X$ and fix a positive integer $l$ and $m \in \mbN ^s$. 
Let $\mcY _i \subset \mcX \times \mbA ^1$ be the closed substack corresponding to $\widetilde{\mfb _i} \mcO _{\mcX \times \mbA ^1}$. 
Take a positive integer $m'$ such that $m' \ge m_i$ holds for any $i$. 
In addition, fix a connected component $\mcV$ of $|\mcJ _{0} (\mcX \times \mbA ^1 / \mbA ^1)|$. 
Consider the following twisted jet stacks and morphisms
\[
|\mcJ _{m'} (\mcX \times \mbA ^1 /\mbA ^1) | 
\xrightarrow{\varphi_{m', 0}} |\mcJ _0 (\mcX \times \mbA ^1/\mbA ^1)| 
\xrightarrow{\varphi_{0,\text{b}}} |\mcX \times \mbA ^1| 
\xrightarrow{g} |X \times \mbA ^1| \xrightarrow{f} |\mbA ^1|, 
\]
where $f$ is the second projection and $g$ is the morphism induced by the resolution $\mcX \to X$. 
Then, we set 
\[
S := 
\varphi_{m',0} ^{-1}
(\mcV) \cap
\varphi_{m' ,\text{b}}  ^{-1} (g^{-1}(Z \times \mbA ^1)) 
\cap 
\bigcap _{1 \le i \le s} \varphi _{m' ,m_i } ^{-1} (|\mcJ _{m_i} (\mcY _i /\mbA ^1)|). 
\]
Let $F$ be the composite map 
$S \hookrightarrow |\mcJ _{m'} (\mcX \times \mbA ^1 /\mbA ^1) |
\to |\mbA ^1|$. 
Then, as in the proof of Theorem \ref{thm:LSC_general}, 
the assertion can reduce to the upper semi-continuity property of the function
\[
|\mbA ^1| \to \mbZ; \qquad p \mapsto \dim F^{-1} (p). 
\]
For each integer $n$, we set
\begin{align*}
S_{\ge n} &:= \{ s \in S \, | \, \dim _s F^{-1} ( F(s)) \ge n  \}, \\
|\mbA ^1|_{\ge n} &:= \{ p \in |\mbA ^1| \, | \, \dim F^{-1} (p) \ge n \}. 
\end{align*}
Then we have $F (S_{\ge n}) = |\mbA ^1|_{\ge n}$. 

By the same reason as in the proof of Theorem \ref{thm:LSC_general}, it follows that
$(g \circ \varphi _{m', \text{b}})(S_{\ge n})$ is closed in $|X \times \mbA ^1|$. 
By definition, $\widetilde{\mfb _i}$ is isotrivial over $\mbA ^1 \setminus \{ 0 \} = \mbC ^*$. 
In addition, we note that $(g \circ \varphi _{m', \text{b}})(S_{\ge n})$ is contained in $|Z \times \mbA ^1|$ and 
that $Z \subset X$ consists of  $\mbC ^*$-fixed points. 
Hence, if $(p, \gamma) \in (g \circ \varphi _{m', \text{b}})(S_{\ge n})$ 
for some $p \in |X|$ and $\gamma \in \mbA ^1 \setminus \{ 0 \}$, 
then $(p, \gamma ') \in (g \circ \varphi _{m', \text{b}})(S_{\ge n})$ 
for any $\gamma ' \in \mbA ^1 \setminus \{ 0 \} $. 
Since $(g \circ \varphi _{m', \text{b}})(S_{\ge n})$ is closed, 
the latter condition implies $(p, 0) \in (g \circ \varphi _{m', \text{b}})(S_{\ge n})$. 
Therefore we can conclude that $F (S_{\ge n})$ is one of $\emptyset$, $\{ 0 \}$ or $\mbA ^1$, 
which completes the proof. 
\end{proof}

\subsection{The toric case}
If $(X, \Delta, \prod _{i=1} ^s \mfa _i ^{r_i})$ is a toric triple and 
$Z$ is a torus invariant closed set, then we can construct a 
$\mbC ^*$-action on $X$  as in Proposition \ref{prop:main2}. 

\begin{proof}[Proof of Theorem \ref{thm:toric}]
By Remark \ref{rmk:property}, we may assume that $X$ is an affine toric variety and $Z$ is irreducible. 
In addition, by Remark \ref{rmk:modification}, we may assume that $X$ is an affine $\mbQ$-factorial toric variety. 
Note that a $\mbQ$-factorial toric variety has a crepant resolution in the category of smooth DM toric stacks. 
As in the proof of Proposition \ref{prop:main2}, we may assume $\Delta = 0$. 

Let $n$ be the dimension of $X$, $M$ a free $\mbZ$-module of rank $n$, and set $M_{\mbR} := M \otimes \mbR$. 
Since $X$ is an affine normal toric variety, we may assume
$X = \Spec \mbC [\sigma \cap M]$ for some rational convex cone $\sigma \subset M_{\mbR}$. 
For $m \in \sigma \cap M$, we denote by $\chi ^m$ the corresponding function in $\mbC [\sigma \cap M]$. 
We call such an element a \textit{monomial}. 
Since $\mfa _i$ and $I_Z$ are torus invariant ideals, they are generated by monomials in $\mbC [\sigma \cap M]$. 

If $Z = \emptyset$, then the assertion is trivial. 
In what follows, we assume $Z \not = \emptyset $. 
Since $Z$ is torus invariant, there exists a face $F$ of $\sigma$ such that 
$I_Z$ is generated by the monomials $\chi ^m$ for $m \in (\sigma \setminus F) \cap M$. 
Since $Z \not = \emptyset$, we can take a hyperplane $H$ such that 
$\sigma \cap H = F$. 
Hence there exists $w \in M^{\vee}$ such that 
$\langle m , w \rangle \ge 1$ for every $m \in (\sigma \setminus F) \cap M$ and that
$\langle m , w \rangle = 0$ for every $m \in F$. 
We fix such $w \in M^{\vee}$. 

The element $w \in M^{\vee}$ provides the $\mbC ^*$-action on the ring $\mbC [\sigma \cap M]$ as follows: 
for $\gamma \in \mbC ^*$
\[
\mbC [\sigma \cap M] \longrightarrow \mbC [\sigma \cap M] ;\qquad 
\chi ^m \mapsto \gamma ^{\langle m , w \rangle} \chi ^m.
\]
Then, this $\mbC ^*$-action is of ray type because $\langle m , w \rangle \ge 0$ for every $m \in \sigma \cap M$. 
In addition, $Z$ is the set of all $\mbC ^*$-fixed points 
because $I_Z = \bigoplus _{m \ge 1} A^{(m)}$.
Since $\mfa _1, \ldots , \mfa _s$ are torus invariant, $\mfa _1, \ldots , \mfa _s$ are $\mbC ^*$-invariant. 
Therefore we can apply Proposition \ref{prop:main2} to the toric pair and complete the proof. 
\end{proof}

\section*{Acknowledgments}
The author expresses his gratitude to his advisor Professor Yujiro Kawamata 
for his encouragement and valuable advice. 
He is grateful to Atsushi Ito and 
Professors Shihoko Ishii, Daisuke Matsushita, Shinnosuke Okawa, Shunsuke Takagi, 
and Yoshinori Gongyo for useful comments and suggestions.
He is supported by the Grant-in-Aid for Scientific Research
(KAKENHI No. 25-3003) and the Grant-in-Aid for JSPS fellows.

\end{document}